\documentclass[conference]{IEEEtran}
\IEEEoverridecommandlockouts

\usepackage{amsmath,amssymb,amsfonts,amsthm}
\usepackage{graphicx}
\usepackage{booktabs}
\usepackage{cite}
\usepackage{bm}
\usepackage{multirow}
\usepackage{array}

\newtheorem{proposition}{Proposition}
\newtheorem{corollary}{Corollary}
\newtheorem{remark}{Remark}

\newcommand{\R}{\mathbb{R}}
\newcommand{\bx}{\mathbf{x}}
\newcommand{\xz}{\mathbf{x}_0}
\newcommand{\blam}{\bm{\lambda}}
\newcommand{\umax}{u_{\max}}
\newcommand{\tf}{t_f}
\newcommand{\ts}{t_s}
\newcommand{\sig}{\sigma}
\newcommand{\ba}{\mathbf{a}}

\begin{document}

\title{On Linear Critical-Region Boundaries in\\
Continuous-Time Multiparametric Optimal Control}

\author{%
  \IEEEauthorblockN{Lida Lamakani\textsuperscript{1,2}\quad
                    Efstratios N.~Pistikopoulos\textsuperscript{1,2}}\\[2pt]
  \IEEEauthorblockA{%
    \textsuperscript{1}Texas A\&M Energy Institute\quad
    \textsuperscript{2}Artie McFerrin Department of Chemical Engineering\\
    Texas A\&M University, College Station, TX~77843, USA\\
    \texttt{lida.lamakani@tamu.edu}\quad\texttt{stratos@tamu.edu}}
}

\maketitle

\begin{abstract}
When an optimal control problem is solved for all possible initial
conditions at once, the initial-state space splits into critical
regions, each carrying a closed-form control law that can be
evaluated online without solving any optimization. This is the
multiparametric approach to explicit control. In the continuous-time
setting, the boundaries between these regions are determined by
extrema of Lagrange multipliers and constraint functions along the
optimal trajectory. Whether a boundary is a hyperplane, computable
analytically, or a curved manifold that requires numerical methods
has a direct effect on how the partition is built.

We show that a boundary is a hyperplane if and only if the relevant
extremum is attained at either the initial time or the terminal time,
regardless of the initial condition. The reason is that the costate
is a linear function of the initial state at any fixed time, so when
the extremum is tied to a fixed endpoint, the boundary condition is
linear and the boundary normal follows directly from two matrix
exponentials and a linear solve. When the extremum occurs at a time
that shifts with the initial condition, such as a switching time or
an interior stationary point, the boundary is generally curved.

We demonstrate the result on a third-order system, obtaining the
complete three-dimensional critical-region partition analytically
for the first time in this problem class. A comparison with a
discrete-time formulation shows how sharply the region count grows
under discretization, while the continuous-time partition remains
unchanged.
\end{abstract}


\section{Introduction}

Multiparametric model predictive control (mp-MPC) solves the optimal
control problem offline for all possible initial conditions and
stores the result as a piecewise affine function of the state
\cite{bemporad2002explicit,tondel2003algorithm,pistikopoulos2020multi,pistikopoulos2007multi}.
Once this offline computation is done, applying the controller
online requires nothing more than finding which region the current
state belongs to and reading off the corresponding affine law.
There is no optimization to solve at runtime, which makes this
approach well suited to systems where computation time or hardware
resources are limited \cite{rawlings2020model}.

The discrete-time formulation of this framework is well established
\cite{tondel2003algorithm,sakizlis2004design,oberdieck2015explicit,oberdieck2017explicit}.
However, as the time grid is refined, the number of critical regions
typically grows combinatorially, since each additional node introduces
new combinations of active constraints.

A continuous-time alternative that avoids this growth was introduced
in \cite{lamakani2026multiparametric}, building on earlier work in \cite{sakizlis2005explicit}.
In this approach, Pontryagin's minimum principle is applied directly,
leading to an explicit partition of the initial-state space without
time discretization.
The resulting boundaries are described by two scalar functions of
$\xz$: $\bar{\mu}_i(\xz)$, the minimum of the multiplier over its
active arc, and $\bar{G}_i(\xz)$, the maximum constraint value over
the portions of the horizon where it is inactive.

In previously reported examples, these boundaries appear to be linear
in $\xz$, but a general explanation has not been established.
In particular, it is not clear when such linearity should be expected
and when curved boundaries arise.

In this paper, we show that a boundary is a hyperplane if and only if
the corresponding extremum is attained at a fixed horizon endpoint.
This provides a simple criterion that distinguishes between boundaries
that admit closed-form expressions and those that require numerical
approximation.

The rest of the paper is organized as follows.
Section~\ref{sec:problem} states the problem and the relevant
optimality conditions.
Section~\ref{sec:geom} proves the main result.
Section~\ref{sec:example} applies it to the third-order system.
Section~\ref{sec:ts} covers the switching-time parametrization.
Section~\ref{sec:dt} presents the discrete-time comparison.
Section~\ref{sec:conclude} closes with a summary and directions for
future work.

\section{Problem Statement}
\label{sec:problem}

\subsection{Optimal Control Problem}

Let $\bx(t)\in\R^n$ denote the state and $u(t)\in\R$ the scalar
control input.
The dynamics are
\begin{equation}
  \dot{\bx}(t) = A\bx(t) + Bu(t), \qquad \bx(0) = \xz,
  \label{eq:sys}
\end{equation}
with $A\in\R^{n\times n}$ and $B\in\R^n$ fixed.
The cost to minimize over a finite horizon $[0,\tf]$ is
\begin{equation}
  J(\xz) = \frac{1}{2}\|\bx(\tf)\|^2
          + \frac{1}{2}\int_0^{\tf}
            \bigl(\|\bx\|^2 + u^2\bigr)\,dt,
  \label{eq:cost}
\end{equation}
subject to constraints that may include a scalar input bound
$|u(t)|\le\umax$ and/or path constraints $g(\bx(t),u(t))\le 0$ for
all $t\in[0,\tf]$.
The initial condition $\xz\in\Theta\subset\R^n$ is treated as a free
parameter, and the aim is to determine the optimal control law
explicitly as a function of $\xz$ over $\Theta$.

\subsection{Pontryagin Conditions and the Switching Function}

The Pontryagin minimum principle \cite{pontryagin2018mathematical} introduces the
costate $\blam(t)\in\R^n$ governed by
\begin{equation}
  \dot{\blam}(t) = -\bx(t) - A^\top\blam(t), \qquad
  \blam(\tf) = \bx(\tf).
  \label{eq:costate}
\end{equation}
The optimal control minimizes the Hamiltonian pointwise in $u$.
For an input bound $|u|\le\umax$ the minimizer is
$u^*(t) = -\operatorname{clip}_{\umax}(\sig(t;\xz))$, where
$\operatorname{clip}_{\umax}$ projects onto $[-\umax,+\umax]$ and the
\emph{switching function} is
\begin{equation}
  \sig(t;\xz) \triangleq B^\top\blam(t;\xz).
  \label{eq:sigma}
\end{equation}
The switching function is named for its role in determining the
arc type at each time: the optimal control switches between a
free arc ($|\sig|\le\umax$) and a bound-active arc
($|\sig|>\umax$) depending on whether $\sig(t;\xz)$ lies inside
or outside $[-\umax,+\umax]$.
The \emph{switching time} $\ts(\xz)$, at which $\sig$ crosses
$\pm\umax$ and the arc type changes, is a separate quantity, namely
an implicit function of $\xz$ defined by the equation
$\sig(\ts(\xz);\xz)=\pm\umax$.
Since \eqref{eq:sys}--\eqref{eq:costate} form a linear system in the
joint state-costate vector $(\bx,\blam)$, both components at any
\emph{fixed} time $\bar{t}$ are linear in $\xz$.
The switching function therefore satisfies
\begin{equation}
  \sig(\bar{t};\xz) = \ba(\bar{t})^\top\xz + b(\bar{t})
  \label{eq:siglinear}
\end{equation}
for a vector $\ba(\bar{t})\in\R^n$ and scalar $b(\bar{t})$ that
depend on $\bar{t}$ but are independent of $\xz$.
This linearity holds at each \emph{fixed} evaluation time $\bar{t}$;
the switching time $\ts(\xz)$ at which $\sig$ crosses $\pm\umax$
is a separate, implicitly defined function of $\xz$ and is not linear.

\subsection{Critical Regions and Boundary Functions}
\label{sec:CR}

A \emph{critical region} $\mathcal{R}\subseteq\Theta$ is a maximal
connected set of initial conditions for which the optimal trajectory
follows the same arc sequence \cite{lamakani2026multiparametric}.
Two conditions characterize its interior: every inactive constraint
remains strictly inactive, and every active Lagrange multiplier
remains strictly positive throughout its active arc.

For constraint $i$ that is active over the arc
$[t_{i,s}^-(\xz),\,t_{i,s}^+(\xz)]$, multiplier positivity is
encoded by
\begin{equation}
  \bar{\mu}_i(\xz) \triangleq
    \min_{\,t\,\in\,[t_{i,s}^-(\xz),\,t_{i,s}^+(\xz)]\,}
    \mu_i(t;\xz) \;>\; 0.
  \label{eq:mubar}
\end{equation}
For constraint $i$ inactive over some portion of the horizon, the
feasibility condition over those inactive portions is
\begin{equation}
  \bar{G}_i(\xz) \triangleq
    \max_{\substack{t\,\in\,[0,\tf] \\
          t\,\notin\,[t_{i,s}^-(\xz),\,t_{i,s}^+(\xz)]}}
    g_i(t;\xz) \;<\; 0,
  \label{eq:Gbar}
\end{equation}
where the maximum excludes the active arc.
When constraint $i$ is inactive over the whole horizon, the excluded
set in \eqref{eq:Gbar} is empty and the maximum is taken over
all of $[0,\tf]$.

The boundaries of $\mathcal{R}$ are the zero level sets
$\{\xz:\bar{\mu}_i(\xz)=0\}$ and $\{\xz:\bar{G}_i(\xz)=0\}$.
Crossing $\{\xz:\bar{\mu}_i(\xz)=0\}$ signals that the multiplier
is losing positivity at some point in the active arc; crossing
$\{\xz:\bar{G}_i(\xz)=0\}$ signals that an inactive constraint
is becoming active.
In either case the arc sequence changes and the optimal solution
belongs to a different critical region.

The optimization intervals in \eqref{eq:mubar} and \eqref{eq:Gbar}
both depend on $\xz$ through the switching times $t_{i,s}^\pm(\xz)$.
Whether this dependence makes the boundary functions nonlinear, and
hence the boundaries curved, is precisely what Proposition~\ref{prop:main}
characterizes.

\section{When Are Boundaries Hyperplanes?}
\label{sec:geom}

\subsection{Main Result}

\begin{proposition}
\label{prop:main}
Let system \eqref{eq:sys} be linear and time-invariant with cost
\eqref{eq:cost} and a pure scalar input bound $|u|\le\umax$.
A critical-region boundary defined by $\{\xz:\bar{\mu}_i(\xz)=0\}$
is a hyperplane in $\xz$ if and only if the minimum in
\eqref{eq:mubar} is attained at a fixed time $\bar{t}\in\{0,\tf\}$
that is independent of $\xz$.
\end{proposition}

\begin{proof}
\textbf{($\Leftarrow$) Endpoint implies hyperplane.}
Suppose $\bar{\mu}_i(\xz)=\mu_i(0;\xz)$ for all $\xz$ near the
boundary (the case $\bar{t}=\tf$ follows by the same argument).
By \eqref{eq:siglinear}, $\sig(0;\xz)=\ba(0)^\top\xz+b(0)$, which
is affine in $\xz$ with coefficients fixed by the system data.
For a pure input bound the multiplier at $t=0$ equals
$\pm\sig(0;\xz)-\umax$, with the sign fixed by which bound is
active, so in both cases $\mu_i(0;\xz)$ is an affine function of $\xz$.
Setting $\bar{\mu}_i(\xz)=\mu_i(0;\xz)=0$ therefore gives
$\ba(0)^\top\xz = c$ for some fixed constant $c$, which is a hyperplane.
The same reasoning applies at $\bar{t}=\tf$.

\textbf{($\Rightarrow$) Non-endpoint implies nonlinearity.}
Suppose the minimum of $\mu_i(t;\xz)$ over the active arc is attained
at a time $\bar{t}(\xz)$ that depends on $\xz$.
Two cases arise.

\emph{Case~1: $\bar{t}(\xz)=t_{i,s}^\pm(\xz)$ is a switching time.}
The switching times depend on $\xz$ through the PMP junction
conditions \cite{lamakani2026multiparametric} and are generically nonconstant.
At a switching time the multiplier transitions between zero and a
positive value, so $\partial_t\mu_i|_{t_{i,s}^\pm}\ne 0$
by the junction conditions of \cite{lamakani2026multiparametric}.
Differentiating $\bar{\mu}_i(\xz)=\mu_i(t_{i,s}^\pm(\xz);\xz)$
by the chain rule:
\begin{equation}
  \nabla_{\xz}\bar{\mu}_i =
  \underbrace{\frac{\partial\mu_i}{\partial t}\bigg|_{t_{i,s}^\pm}}_{\ne\,0}
  \!\cdot\,\nabla_{\xz}t_{i,s}^\pm
  \;+\; \ba\!\bigl(t_{i,s}^\pm(\xz)\bigr).
  \label{eq:chainA}
\end{equation}
Both terms on the right depend on $\xz$ through $t_{i,s}^\pm(\xz)$,
so $\nabla_{\xz}\bar{\mu}_i$ is not constant and the boundary is not
a hyperplane.

\emph{Case~2: $\bar{t}(\xz)$ is an interior stationary point.}
A necessary condition for an interior minimum of $\mu_i(t;\xz)$ is
$\partial_t\mu_i(\bar{t};\xz)=0$.
For a pure input bound this reduces to $\dot{\sig}(\bar{t};\xz)=0$.
Since $\dot{\sig}(t;\xz)$ is linear in $\xz$ at each fixed $t$,
the implicit function theorem gives $\bar{t}(\xz)$ as a smooth
function of $\xz$.
Differentiating $\bar{\mu}_i(\xz)=\mu_i(\bar{t}(\xz);\xz)$:
\begin{equation}
  \nabla_{\xz}\bar{\mu}_i =
  \underbrace{\frac{\partial\mu_i}{\partial t}\bigg|_{\bar{t}}}_{=\,0}
  \!\cdot\,\nabla_{\xz}\bar{t}
  \;+\; \ba\!\bigl(\bar{t}(\xz)\bigr).
  \label{eq:chainB}
\end{equation}
The first term vanishes.
The remaining term $\ba(\bar{t}(\xz))$ is the coefficient vector of
$\sig$ at the time $\bar{t}(\xz)$, which varies with $\xz$, so
$\nabla_{\xz}\bar{\mu}_i$ is nonconstant and the boundary is not a
hyperplane.
\end{proof}

\begin{remark}
\label{rem:general}
The proof relies on \eqref{eq:siglinear}, the linearity of the
switching function at each fixed time, which holds for any linear
time-invariant system regardless of state dimension or number of arc
transitions.
For problems with path or output constraints, the multiplier
$\mu_i(t;\xz)$ at any fixed time is similarly linear in $\xz$
through the costate dynamics, so the same endpoint condition governs
whether the corresponding boundary is a hyperplane.
The present paper focuses on the input-bound case, where the full
proof is given; the extension to path constraints follows by an
analogous argument applied to the respective boundary functions
\eqref{eq:Gbar}.
\end{remark}

\subsection{The Single-Switch Case and Direct Computation}

\begin{corollary}
\label{cor:single}
For a linear time-invariant system with a scalar input bound
$|u|\le\umax$ and at most one arc transition per optimal trajectory,
every critical-region boundary is a hyperplane.
\end{corollary}

\begin{proof}
With at most one transition, the possible arc sequences are a
full-horizon free arc, a full-horizon bound-active arc, or a
bound-active arc $[0,\ts(\xz)]$ followed by a free arc
$[\ts(\xz),\tf]$.
Two boundary types arise.

For the boundary between the free-arc region and the transitional
regions, the bound-active arc entry time is $t_{i,s}^-(\xz)=0$, which
is a fixed horizon endpoint.
At the arc entry the multiplier transitions from zero to positive;
it satisfies $\mu_i(0;\xz)=0$ exactly on the boundary and
$\mu_i(t;\xz)>0$ for $t\in(0,\ts(\xz))$ in the interior of the
region, by the junction conditions at arc entry \cite{lamakani2026multiparametric}.
The minimum of $\mu_i$ over the active arc $[0,\ts(\xz)]$ is therefore
attained at the fixed endpoint $t=0$, and
Proposition~\ref{prop:main} gives a hyperplane.

For the boundary between the transitional regions and the full-horizon
bound-active region, the active arc covers all of $[0,\tf]$ at the
boundary, so $t_{i,s}^+(\xz)=\tf$ is a fixed endpoint.
By the junction conditions at arc exit, $\mu_i(\tf;\xz)=0$ on the
boundary and $\mu_i(t;\xz)>0$ for $t\in(0,\tf)$ in the interior,
so the minimum of $\mu_i$ over $[0,\tf]$ is attained at $t=\tf$.
Proposition~\ref{prop:main} again gives a hyperplane.
\end{proof}

Corollary~\ref{cor:single} leads directly to a computation procedure.
Partition the augmented matrix exponentials as
\begin{equation}
  \Phi_f = e^{M_f\tf} =
  \begin{bmatrix}\Phi_f^{11}&\Phi_f^{12}\\
                  \Phi_f^{21}&\Phi_f^{22}\end{bmatrix},\quad
  \Phi_s = e^{M_s\tf} =
  \begin{bmatrix}\Phi_s^{11}&\Phi_s^{12}\\
                  \Phi_s^{21}&\Phi_s^{22}\end{bmatrix},
  \label{eq:phiparts}
\end{equation}
where $M_f,M_s\in\R^{2n\times 2n}$ are the Hamiltonian system matrices
for the free and bound-active arcs respectively, each partitioned into
$n\times n$ blocks.
Solving the free-arc two-point boundary value problem yields
$\blam(0)=K_f\xz$ with
\begin{equation}
  K_f = -\bigl(\Phi_f^{22}-\Phi_f^{12}\bigr)^{-1}
          \bigl(\Phi_f^{21}-\Phi_f^{11}\bigr),
  \label{eq:Kf}
\end{equation}
from which the first pair of boundary normals is
$\ba_f = K_f^\top B \in \R^n$.
For the example in Section~\ref{sec:example} where $B=e_n$, this
simplifies to $\ba_f=[K_f]_{n,:}^\top$, the last row of $K_f$.
The bound-active TPBVP gives $\sig(\tf;\xz)=\ba_s^\top\xz+c_s$ by an
analogous computation, with $\ba_s = K_s^\top B$, yielding the second pair.
The entire procedure involves two matrix exponentials and two linear
solves; no time integration or root-finding is required.
This gives a complete analytical algorithm for computing every
critical-region boundary in problems satisfying
Corollary~\ref{cor:single}, with a computational cost that depends
only on the system order and is independent of the size or shape of
the parameter set $\Theta$.

\section{Third-Order Demonstration}
\label{sec:example}

\subsection{System and Parameter Space}

We apply the result to the third-order system
\begin{equation}
  A = \begin{bmatrix}0&1&0\\0&0&1\\-2&-2&-5\end{bmatrix},
  \quad
  B = \begin{bmatrix}0\\0\\1\end{bmatrix},
  \label{eq:AB}
\end{equation}
with $\tf=5$, $\umax=0.4$, and parameter box
$\Theta=[-2.6,2.6]\times[-0.9,0.9]\times[-0.7,0.7]$.
All eigenvalues of $A$ have strictly negative real parts, so the
open-loop system is asymptotically stable.
The input enters through the third state alone, and all three
initial-state components are free parameters.
This extends the one- and two-dimensional examples of \cite{lamakani2026multiparametric}
to a three-dimensional parameter space for the first time.

Numerical integration of the optimality system confirms that
$\sig(t;\xz)$ crosses $\pm\umax$ at most once per trajectory for every
$\xz\in\Theta$, so Corollary~\ref{cor:single} applies: every
critical-region boundary is a hyperplane.

\subsection{Boundary Computation and Verification}

Table~\ref{tab:arcs} lists the five critical regions and identifies
where $\bar{\mu}_i$ is minimized for each boundary, confirming the
endpoint condition of Proposition~\ref{prop:main} in every case.

\begin{table}[t]
\centering
\caption{CT critical regions and their boundaries.
``BA'' = bound active; ``$\to$'' = transition to free arc.
The third column gives the time at which the boundary-defining
extremum of $\bar{\mu}_i$ is attained, confirming the endpoint
condition of Proposition~\ref{prop:main} for each boundary.}
\label{tab:arcs}
\renewcommand{\arraystretch}{1.08}
\begin{tabular}{@{}lll@{}}
\toprule
Region & Arc sequence & $\bar{\mu}_i$ at \\
\midrule
CR01 & Free arc, $[0,\tf]$        & $t=0$ \\
CR02 & Upper BA $[0,\ts]\to$ free & $t=0$ \\
CR03 & Upper BA, $[0,\tf]$        & $t=\tf$ \\
CR04 & Lower BA $[0,\ts]\to$ free & $t=0$ \\
CR05 & Lower BA, $[0,\tf]$        & $t=\tf$ \\
\bottomrule
\end{tabular}
\end{table}

\textbf{Boundaries $\ell_1$, $\ell_2$} (between CR01 and CR02/CR04).
On a free arc, $\blam(0)=K_f\xz$ from \eqref{eq:Kf}.
With $B=[0,0,1]^\top$, the boundary normal is
$\ba_f = K_f^\top B = [K_f]_{3,:}^\top \in \R^3$
and the switching function at $t=0$ satisfies $\sig(0;\xz)=\ba_f^\top\xz$.
The boundary condition $|\sig(0;\xz)|=\umax$ gives
\begin{equation}
  \ell_{1,2}:\quad \ba_f^\top\xz = \pm\umax.
  \label{eq:l12}
\end{equation}

\textbf{Boundaries $\ell_3$, $\ell_4$} (between CR02/CR04 and CR03/CR05).
Under a full-horizon bound-active trajectory, the bound-active TPBVP
gives $\sig(\tf;\xz)=\ba_s^\top\xz+c_s$ where $\ba_s = K_s^\top B$.
The boundary condition $|\ba_s^\top\xz+c_s|=\umax$ at $t=\tf$ gives
\begin{equation}
  \ell_{3,4}:\quad \ba_s^\top\xz = \mp\umax - c_s.
  \label{eq:l34}
\end{equation}

Computing $e^{M_f\tf}$ and $e^{M_s\tf}$ and applying \eqref{eq:Kf}:
\begin{align}
  \ell_1:&\quad 0.2084\,x_{01}+0.8613\,x_{02}+0.2650\,x_{03}=-0.4,
  \label{eq:L1}\\
  \ell_2:&\quad 0.2084\,x_{01}+0.8613\,x_{02}+0.2650\,x_{03}=+0.4,
  \label{eq:L2}\\
  \ell_3:&\quad 0.1944\,x_{01}+0.1593\,x_{02}+0.0252\,x_{03}=-0.361,
  \label{eq:L3}\\
  \ell_4:&\quad 0.1944\,x_{01}+0.1593\,x_{02}+0.0252\,x_{03}=+0.361.
  \label{eq:L4}
\end{align}
The dominant coefficient on $x_{02}$ in $\ba_f$ reflects the strong
coupling of the second state into the switching function through the
costate dynamics; the small third coefficient in $\ba_s$ follows from
the rapid attenuation of $x_{03}$'s influence by $t=\tf$.

\subsection{Cube-Face Visualization}

Figure~\ref{fig:ct3d} renders the partition by coloring the six faces
of $\Theta$ according to the critical region they belong to.
The boundary edges visible on every face are intersections of the four
hyperplanes \eqref{eq:L1}--\eqref{eq:L4} with the respective face.
Since a hyperplane in $\R^3$ meets any flat surface in a straight line,
the straightness of those edges is a geometric consequence of
Proposition~\ref{prop:main}: if any boundary were curved in the
interior of $\Theta$, its trace on the cube face would be curved too.
The figure thus confirms the analytical result visually.

\begin{figure}[t]
  \centering
  \includegraphics[width=0.90\columnwidth]{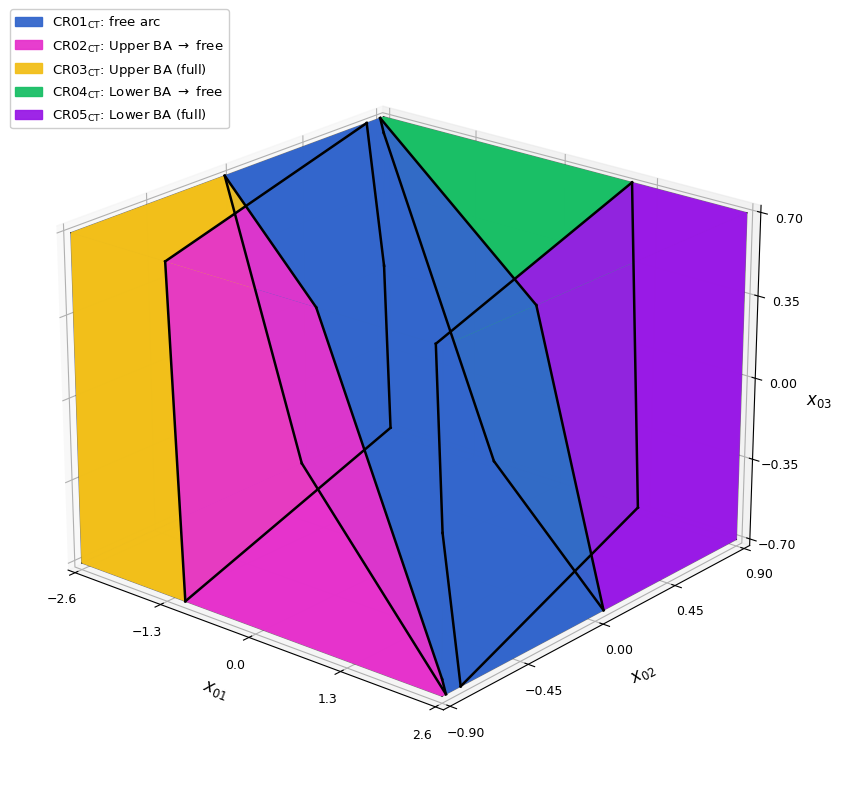}
  \caption{CT critical-region partition of
    $\Theta=[-2.6,2.6]\times[-0.9,0.9]\times[-0.7,0.7]$.
    Cube faces are colored by critical region; black lines are
    cross-sections of the four hyperplanes \eqref{eq:L1}--\eqref{eq:L4}.
    Straight edges on every face confirm the endpoint condition of
    Proposition~\ref{prop:main}.}
  \label{fig:ct3d}
\end{figure}

\section{Switching-Time Parametrization}
\label{sec:ts}

Within CR02 and CR04 the trajectory switches from a bound-active arc to
a free arc at a time $\ts(\xz)\in(0,\tf)$.
This transition instant, defined by $\sig(\ts;\xz)=\mp\umax$, is
distinct from the time at which $\bar{\mu}_i$ is minimized: by
Corollary~\ref{cor:single} the minimum occurs at the fixed endpoint
$t=0$, while $\ts$ is the interior time at which the arc structure
changes.
Because $\sig(t;\xz)$ crosses $\pm\umax$ transversally at $\ts$,
meaning $\dot{\sig}(\ts;\xz)\ne 0$ generically, bisection on
the scalar equation $\sig(t;\xz)=\mp\umax$ converges reliably for
any fixed $\xz$ and forms the basis of the numerical computation.
Since $\ts(\xz)$ is a continuous function of $\xz$ within each
region \cite{lamakani2026multiparametric}, it admits a polynomial approximation.
We compute it by bisection at sampled initial conditions and fit a
polynomial of total degree at most three in $(x_{01},x_{02},x_{03})$
by least squares.

\begin{table}[t]
\centering
\caption{Switching-time polynomial fits (degree~$3$).}
\label{tab:ts}
\renewcommand{\arraystretch}{1.08}
\begin{tabular}{@{}lccc@{}}
\toprule
Region & $R^2$ & $\ts^{\min}$ (s) & $\ts^{\max}$ (s) \\
\midrule
CR02 & 0.9892 & 0.016 & 2.689 \\
CR04 & 0.9919 & 0.010 & 2.505 \\
\bottomrule
\end{tabular}
\end{table}

Table~\ref{tab:ts} shows that degree-3 polynomials achieve $R^2>0.989$
in both regions.
All three initial-state components enter the fits with non-negligible
coefficients, reflecting the coupling that $A$ introduces among the
states during the bound-active arc.
The wide range of $\ts$, from under 0.02 seconds to nearly 2.7
seconds within a five-second horizon, shows that the transition can
occur anywhere from very early to well past the midpoint of the
planning window.

\section{Discrete-Time Comparison}
\label{sec:dt}

We solve the same problem as a discrete-time multiparametric quadratic
program to provide context for the CT results.
System~\eqref{eq:sys} is discretized by exact zero-order hold at step
size $T_s=\tf/N$, the stage cost is integrated exactly over each
interval, and the constraint $|u_k|\le\umax$ is enforced at every node.
PPOPT \cite{kenefake2022ppopt}, which implements the connected-graph
algorithm of \cite{oberdieck2017explicit}, is used to solve the resulting
multiparametric QP.

Table~\ref{tab:compare} reports the region counts for two grid sizes,
and Figures~\ref{fig:dt5} and \ref{fig:dt10} display the partitions on
the same cube as Fig.~\ref{fig:ct3d}.
Table~\ref{tab:affine} shows the affine control law for the largest
critical region at $N=5$, illustrating the form of the DT explicit
solution.

\begin{table}[t]
\centering
\caption{Critical-region counts: CT vs.\ DT on the same system
and parameter box.}
\label{tab:compare}
\renewcommand{\arraystretch}{1.08}
\begin{tabular}{@{}lcc@{}}
\toprule
Formulation & Grid $N$ & Regions \\
\midrule
Continuous-time & --- & \textbf{5} \\
\multirow{2}{*}{Discrete-time} & 5  & 23 \\
                                & 10 & 77 \\
\bottomrule
\end{tabular}
\end{table}

\begin{table}[t]
\centering
\caption{Affine control law $u_k=K_k\xz$ for the largest DT
critical region, $N=5$.}
\label{tab:affine}
\renewcommand{\arraystretch}{1.05}
\setlength{\tabcolsep}{5pt}
\begin{tabular}{@{}crrr@{}}
\toprule
$k$ & $x_{01}$ & $x_{02}$ & $x_{03}$ \\
\midrule
0 & $-0.0092$ & $-0.6397$ & $-0.1533$ \\
1 & $+0.2076$ & $-0.2382$ & $-0.0589$ \\
2 & $+0.2322$ & $+0.0407$ & $-0.0014$ \\
3 & $+0.1381$ & $+0.1330$ & $+0.0220$ \\
4 & $-0.0055$ & $+0.0428$ & $+0.0092$ \\
\bottomrule
\end{tabular}
\end{table}

\begin{figure}[t]
  \centering
  \includegraphics[width=0.90\columnwidth]{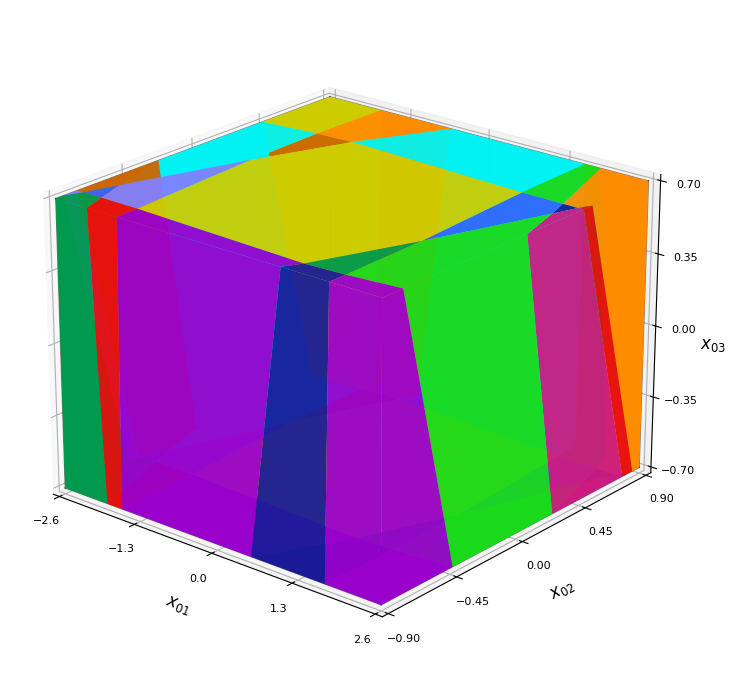}
  \caption{DT partition at $N=5$: 23 critical regions on the same cube.
    Each region corresponds to a distinct active-constraint pattern
    across the 5 grid nodes.}
  \label{fig:dt5}
\end{figure}

\begin{figure}[t]
  \centering
  \includegraphics[width=0.90\columnwidth]{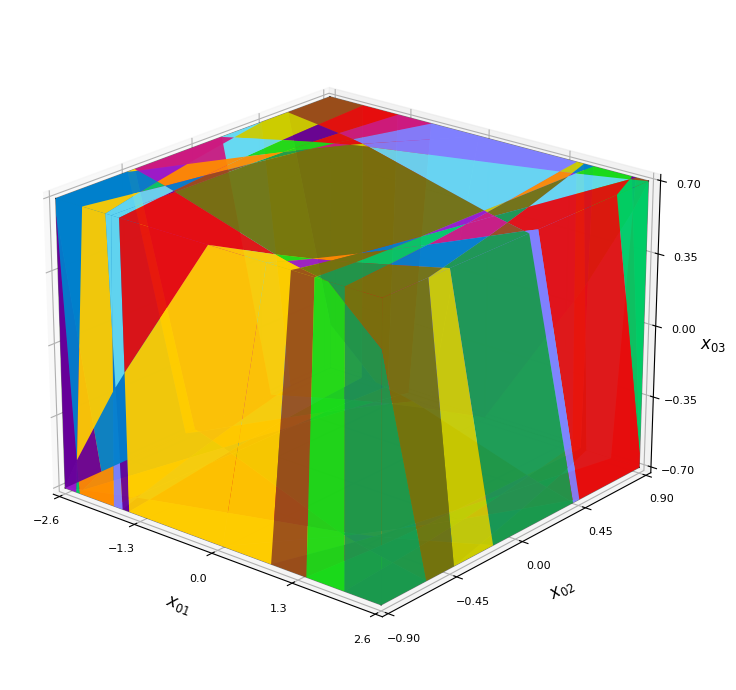}
  \caption{DT partition at $N=10$: 77 critical regions.
    The region count grows as more nodes bring new active-set
    combinations.}
  \label{fig:dt10}
\end{figure}

The DT boundaries are faces of the KKT active-set polytopes of the
finite-dimensional quadratic program and have no counterpart in the
continuous-time boundary functions \eqref{eq:mubar}--\eqref{eq:Gbar};
Proposition~\ref{prop:main} does not apply to them.
The growth in region count with the grid reflects the increasing number
of nodal active-set combinations, a purely combinatorial effect that
is unrelated to boundary geometry.
The CT partition has fewer regions because it works directly with the
differential equations without imposing a temporal grid, not because
its boundaries are hyperplanes.
The hyperplanarity is a separate property: it means the CT boundaries
in this example are available analytically from matrix exponentials,
whereas the DT boundaries require solving the full multiparametric QP.
Table~\ref{tab:affine} also illustrates a structural difference in the
solutions themselves: the CT optimal law within each region is a
single affine expression in $\xz$ valid over the entire arc, while the
DT law specifies a separate gain vector $K_k$ at each of the $N$ nodes,
so the number of stored coefficients grows as $N$ increases even
within a single region.

\section{Conclusion}
\label{sec:conclude}

We have proved that a critical-region boundary in continuous-time
multiparametric optimal control with a scalar input bound is a
hyperplane if and only if the minimum of the active Lagrange multiplier
$\bar{\mu}_i$ over its active arc is attained at a fixed horizon
endpoint independent of the initial condition.
When this holds, the boundary normal is computable from two matrix
exponentials and a linear solve, with no search over the time axis.
When it fails, because the minimum occurs at a switching time or at
an interior stationary point that moves with the initial condition,
the boundary is generically curved and requires numerical approximation.
The same endpoint principle extends to path and output constraints
through an analogous argument, as discussed in Remark~\ref{rem:general}.

The third-order demonstration gives the first three-dimensional
critical-region partition for this problem class, with four boundary
hyperplanes computed analytically and verified through the cube-face
visualization.
The side-by-side discrete-time results show how rapidly the partition
grows when the problem is discretized, while the CT partition remains
compact regardless of accuracy requirements.
Hyperplanarity is not the cause of that compactness, but it is an
additional benefit that makes the CT boundaries available in closed form
when the endpoint condition holds.

Problems with path constraints, output constraints, or multiple arc
transitions per trajectory will generally produce boundaries where the
endpoint condition fails, and Proposition~\ref{prop:main} predicts that
those boundaries are curved.
Approximating them requires locating the interior stationary points
where $\dot{\sig}(t^*;\xz)=0$, computing the switching-time sensitivity
$\nabla_{\xz}t^*$ via the implicit function theorem, and tracing the
resulting nonlinear boundary manifold by predictor-corrector continuation
or IFT-based sigma-tube certification.
Developing and certifying these methods for the multi-switch and
nonlinear boundaries is the immediate priority.
The endpoint condition of Proposition~\ref{prop:main} serves as the
precise criterion that determines, before any computation, which
boundaries fall into the closed-form class and which require this
numerical treatment.

\bibliographystyle{IEEEtran}
\bibliography{refs}

\end{document}